\documentclass{article}
\usepackage[nottoc]{tocbibind}
\usepackage{graphicx}
\usepackage{mathrsfs}
\usepackage{geometry}
\usepackage{titletoc}
\usepackage{amsmath}
\usepackage{extarrows}
\allowdisplaybreaks[4]
\usepackage{amssymb}
\usepackage{bbm}
\usepackage{bm}
\usepackage{mathtools}
\usepackage{enumerate}
\usepackage{amsthm}
\usepackage[font=small]{caption}
\usepackage{subcaption}
\usepackage{cases}
\usepackage{mleftright}
\mleftright
\usepackage[backref]{hyperref} 
\hypersetup{
hidelinks
}
\usepackage{verbatim}

\newlength{\bibitemsep}
\newlength{\bibparskip}
\let\oldthebibliography\thebibliography
\renewcommand\thebibliography[1]{%
  \oldthebibliography{#1}%
  \setlength{\parskip}{\bibitemsep}%
  \setlength{\itemsep}{\bibparskip}%
}
\usepackage[table]{xcolor}
\usepackage[all]{xy}
\usepackage{tikz}

\newcommand\cS{{\mathcal S}}

\newcommand{\VolL}{\mathrm{Vol}\,L}

\newcommand{\ord}{\mathrm{ord}}

\newcommand{\Bc}{\mathrm{Bc}\,}

\newcommand{\Val}{\mathrm{Val}}

\newcommand{\CC}{\mathbb {C}}

\newcommand{\NN}{{\mathbb N}}
\newcommand{\PP}{{\mathbb P}}
\newcommand{\QQ}{{\mathbb Q}}
\newcommand{\RR}{{\mathbb R}}

\DeclareMathOperator{\Aut}{Aut}

\DeclareMathOperator{\Bl}{Bl}

\theoremstyle{plain}
\newtheorem{theorem}{Theorem}[section]
\newtheorem{proposition}[theorem]{Proposition}

\newtheorem{lemma}[theorem]{Lemma}

\newtheorem{claim}[theorem]{Claim}
\newtheorem{corollary}[theorem]{Corollary}

\newtheorem{conjecture}[theorem]{Conjecture}

\newtheorem{problem}[theorem]{Problem}

\theoremstyle{definition}
\newtheorem{definition}[theorem]{Definition}
\newtheorem{remark}[theorem]{Remark}

\def\KE{K\"ahler--Einstein }
\newcommand{\beq}{\begin{equation}}
\newcommand{\eeq}{\end{equation}}
\newcommand{\bpf}{\begin{proof}}
\newcommand{\epf}{\end{proof}}
\newcommand{\baligned}{\begin{aligned}}
\newcommand{\ealigned}{\end{aligned}}
\newcommand{\bdefn}{\begin{definition}}
\newcommand{\edefn}{\end{definition}}
\newcommand{\bremark}{\begin{remark}}
\newcommand{\eremark}{\end{remark}}
\newcommand{\bconj}{\begin{conjecture}}
\newcommand{\econj}{\end{conjecture}}
\newcommand{\bcor}{\begin{corollary}}
\newcommand{\ecor}{\end{corollary}}
\newcommand{\blem}{\begin{lemma}}
\newcommand{\elem}{\end{lemma}}
\newcommand{\bclaim}{\begin{claim}}
\newcommand{\eclaim}{\end{claim}}
\newcommand{\bprob}{\begin{problem}}
\newcommand{\eprob}{\end{problem}}
\newcommand{\bprop}{\begin{proposition}}
\newcommand{\eprop}{\end{proposition}}
\newcommand{\bthm}{\begin{theorem}}
\newcommand{\ethm}{\end{theorem}}
\def\lb#1{\label{#1}}
\def\ra{\rightarrow}

\def\q{\quad}

\def\disp{\displaystyle}

\def\fin{\operatorname{fin}}
\def\Div{\operatorname{div}}
\def\Val{\operatorname{Val}}
\def\osc{\operatorname{osc}}
\newcommand{\ValX}{\Val_X}
\newcommand{\ValXfin}{{\Val}_X^{\fin}}
\newcommand{\ValXdiv}{{\Val}_X^{\Div}}
\newcommand{\bi}{\bibitem}

\title{
Stability thresholds for big classes}
\author{Chenzi Jin, Yanir A. Rubinstein, Gang Tian}
\date{January 25, 2025}

\begin{document}

\maketitle

\begin{abstract}
In 1987, the $\alpha$-invariant theorem gave a fundamental
criterion for existence of \KE metrics on smooth Fano manifolds.
In 2012, Odaka--Sano extended the framework to $\QQ$-Fano
varieties in terms of K-stability, and in 2017 Fujita
related this circle of ideas to the $\delta$-invariant of Fujita--Odaka.
We introduce new invariants on the big cone and prove
a generalization of the Tian--Odaka--Sano Theorem
to all big classes on varieties with klt singularities, and
moreover for all volume quantiles $\tau\in[0,1]$.
The special degenerate (collapsing) case $\tau=0$ on ample classes
recovers Odaka--Sano's theorem. 
This leads to many new twisted \KE metrics on big classes.
Of independent interest,
the proof involves a generalization to {\it sub-barycenters}
of the classical Neumann--Hammer Theorem from convex geometry.

\end{abstract}

 \tableofcontents

\section{Introduction}

Recently, we introduced volume quantiles on the big cone of a projective variety
with klt singularities \cite{JRT}. 
Using these quantiles and their discrete versions we introduced new 
asymptotic algebraic geometric invariants $\bm{\delta}_\tau$ that generalize the log canonical
threshold  $\alpha$ (Tian's invariant \cite{Tian87}) 
and the stability threshold $\delta$ (Fujita--Odaka invariant \cite{FO18}).
These are defined for each $\tau\in[0,1]$ 
in terms of joint limits in $(k,m_k)$ where $k$ is the power 
of the line bundle and $m_k\in\{0,\ldots,d_k\}$ with $d_k:=\dim H^0(X,L^k)$
and $\tau=\lim_k m_k/d_k$ is the percentage of linearly independent 
sections used. In this way, $\alpha=\bm{\delta}_0$ corresponds to the degenerate,
or { collapsing},
case $\tau=0$
while $\delta=\bm{\delta}_1$ to the full mass, or birational, case $\tau=1$
(where a full basis is used).
This point of view can be used
to obtain asymptotics of $\alpha_k$ and $\delta_k$ invariants 
using Okounkov bodies and their discretizations, convexity, and blow-up analysis
inspired by Riemannian convergence theory \cite{JRT}.

In this paper, we wish to give another novel application of the invariants
$\bm{\delta}_\tau$, namely to the existence of \KE metrics and K-stability
in the most general setting possible, i.e., big classes (with some of
the results new even for ample classes). 
An important motivation for this is the recent work of Darvas--Zhang that shows
that K-stability implies the existence of twisted \KE metrics
also on big classes \cite{DarvasZhang}.
One of our main results 
is a generalized Tian--Odaka--Sano theorem, giving 
new valuative criteria for K-stability {\it on the big cone.}
These criteria are in terms of any volume quantile, and are new for all $\tau\in[0,1)$
but there is one unexpected, quite surprising, finding here: in the familiar case $\tau=0$ corresponding
to $\alpha$, {\it a new invariant/criteria emerges, different from $\alpha$}.
We call this new invariant $\tilde\alpha$. On the ample cone $\alpha=\tilde\alpha$
and we recover the Odaka--Sano theorem.

These criteria are {\it not} in terms of 
$\bm{\delta}_\tau$ but rather in terms of a new family of invariants
$\bm{\tilde\delta}_\tau$ that only agree with $\bm{\delta}_\tau$
on the nef cone when $\tau\in[0,1)$.
Remarkably, by taking the ``collapsing limit" $\tau\ra0$ appropriately in our sub-barycenter estimate,
the new invariant $\tilde\alpha$ emerges, which is different from the $\alpha$-invariant outside of the nef cone. 
It is related to the $\alpha$-invariant via a correction term involving
Nakayama's minimal vanishing order $\sigma$. Given the ubiquity of 
the $\alpha$-invariant in many estimates in complex geometry
as well as 
its central r\^ole in birational geometry \cite{Birkar} and in
exceptional quotient singularity theory \cite{CS-GT,CS-GT-6,CS-GT-9}, we believe
this new invariant $\tilde\alpha$ is of importance even beyond the existence
of singular \KE metrics.

Another motivation, in fact our original point of departure, was
to generalize Fujita's beautiful estimates relating $\alpha$ and $\delta$ \cite{Fuj19,F19}.
Fujita relies on an old and classical theorem of Neumann--Hammer \cite{Neumann,Hammer} from convex geometry that
gives an estimate on projections of barycenters of convex bodies 
(applied to Okounkov
bodies). It is a very interesting question of independent interest in convex geometry
whether these estimates generalize. We find such a generalization, to what we term
as {\it sub-barycenters}.
Roughly, imagine sliding a supporting hyperplane to the convex body $K$
and consider the sub-bodies $K_t$ it cuts out inside $K=K_0$; the sub-barycenters are the barycenters
of these sub-bodies, and $\tau=\tau(t)$, the volume quantile alluded to earlier is
the generalized inverse of the function $t\mapsto|K_t|/|K|$.

We end the paper with an explicit computation in the case of a cubic surface
with an Eckardt point. This case was the centerpiece of Tian's proof of the Calabi
Conjecture for del Pezzo surfaces \cite{Tian90-2} as it stood out as the most difficult del Pezzo surface to handle. It has generated many approaches over the past 35 years. We show how the quantile approach gives an elegant and conceptual alternative approach to a key step in this case.

\subsection{Valuative coercivity criteria on the ample cone}

A model result, going back on the analytic side to Tian in 1987,
is the following \cite[Theorem 2.1]{Tian87} (extended by Demailly--Koll\'ar to orbifolds
\cite{DK01}), stated in terms of an equivariant version of $\alpha$:

\bthm
\lb{TianThm} {\rm (Tian 1987)}
Let $X$ be a smooth $n$-dimensional Fano manifold, i.e., $-K_X$ is ample. Let $G$ be a compact
subgroup of $\Aut(X)$.
Then
$\alpha_G {>} \frac n{n+1}
$
implies the existence of a \KE metric.
\ethm

An algebraic counterpart of this result was missing for many years,
and was ultimately established by Odaka--Sano in 2012,
and subsequently refined 
 by Fujita and Blum--Jonsson 
 \cite[Theorem 1.4]{OS},\cite[Proposition 2.1]{Fuj19},~\cite[Theorem A]{BJ20}.

\bthm
\lb{OdakaSanoThm} {\rm (Odaka--Sano 2012)}
Let $X$ be a normal
projective variety of dimension $n$
with klt singularities and $-K_X$ ample.
Then
$\alpha \stackrel{(\ge)}{>} \frac n{n+1}
$
implies K-(semi)stability.
\ethm

Despite multiple generalizations, it has been an open problem whether
Theorems \ref{TianThm}--\ref{OdakaSanoThm} extend to big classes since all known proofs
use ampleness crucially. It is natural to ask:

\bprob
\lb{TianOdakaSanoProb}
Does Tian--Odaka--Sano's criterion generalize to the big cone?
\eprob

A major motivation is that since 1987 and until the introduction of 
the $\delta$-invariant in recent years \cite{FO18}, these theorems
were the main tool for constructing \KE metrics. Even these days,
there are many situations were the $\delta$-invariant is very difficult
to compute already for ample classes, whilst the $\alpha$-invariant is considerably simpler.
Finally, the $\alpha$-invariant holds a central and special place in birational geometry
in its own right, also known algebraically as the global log canonical thresholds (glct)
and it plays a central r\^ole in the study of 
boundedness questions in birational geometry \cite{Birkar}, and the theory of
exceptional quotient singularities \cite{CS-GT,CS-GT-6,CS-GT-9}.

One of our key observations is that the Tian--Odaka--Sano criterion does {\it not}
extend to the big cone. Rather surprisingly, one needs to introduce a new invariant,
that we denote $\tilde\alpha$ for which the criterion holds for all big classes.



\subsection{A new \texorpdfstring{$\alpha$}{alpha}-invariant on the big cone}

Traditionally, the study of log canonical thresholds \`a la Cheltsov
and collaborators \cite{Cheltsov08,CS08} revolves around the maximal vanishing order, or
equivalently, in the analytic language: finding
the worst section(s) in the sense that it has the highest order of vanishing
along a given divisor (this often occurs in a birational model over $X$).
For our new invariant on the big cone, however, we will also need
to keep track of the {\it minimal} vanishing order. This is quite surprising,
as in some sense it corresponds to the best section(s).

\begin{definition}\lb{DivValDef}
    For any proper birational morphism $\pi:Y\to X$ where $Y$ is normal, a prime divisor $F$ on $Y$ is called a {\it divisor over $X$}. It determines a discrete valuation (i.e., integer valued) $\ord_F$ on the function field $\CC(Y)\cong\CC(X)$ (i.e., field of rational functions on $X$), called the {\it vanishing order} along $F$. See \cite[\S II.6]{Har}. A {\it divisorial valuation} $v$ is in the form of $v=v_{C,F}:=C\,\ord_F$ for some $C>0$ and $F$ over $X$. For a section $s\in H^0(X,L)$ of some line bundle $L$,
    $$
        v_{C,F}\left(s\right):=C\,\ord_F\left(\pi^*s\right).
    $$
    Let $\ValXdiv$ denote the set of all divisorial valuations.
\end{definition}

\bdefn
Fix a big line bundle $L$. For $k\in\NN$, let $R_k:=H^0(X,kL)$. Denote $\sigma:\ValXdiv\ra\RR_+$ 
the minimal vanishing order, i.e., $\sigma(v)=\inf_k\inf_{s\in R_k}v(s)/k$,
by  $\cS_0:\ValXdiv\ra\RR_+$ 
the maximal vanishing order, i.e., $\cS_0(v)=\sup_k\sup_{s\in R_k}v(s)/k$,
and by $A:\ValXdiv\ra\RR_+$
the log discrepancy, i.e., for any prime divisor $F$ on $Y$ with $\pi:Y\to X$,
$$
    A\left(\ord_F\right):=1+\ord_F\left(K_Y-\pi^*K_X\right),
$$
and
$$
    A\left(C\,\ord_F\right)=C\cdot A\left(\ord_F\right).
$$
\edefn

\bdefn
\label{tildealphaDef}
Set
$$\tilde\alpha
:=
            \displaystyle\inf_{v\in \ValXdiv} \frac A{\cS_0+\sigma/n}.
$$
\edefn

Note that $\sigma\equiv0$ for ample line bundles \cite[Proposition 2.1]{Fuj19}
so $\alpha=\tilde\alpha$ for nef classes.

We resolve Problem \ref{TianOdakaSanoProb} by the following theorem,
generalizing Theorem \ref{OdakaSanoThm}.

\bthm
\lb{AlphaBigThm}
Let $X$ be a normal
projective variety of dimension $n$
with klt singularities and $-K_X$ big.
Then
$\tilde\alpha \stackrel{(\ge)}{>} \frac n{n+1}
$
implies K-(semi)stability.
\ethm

\subsection{Coercivity criteria on the big cone for all quantiles}

Our next result simultaneously extends Theorem \ref{AlphaBigThm} to all volume
quantiles $\tau\in[0,1]$ and to all big $L$ (not just $L=-K_X$). 
Via the work of Darvas--Zhang \cite{DarvasZhang} this can
be used to construct many new singular twisted \KE metrics on big classes.

A key ingredient are the invariants $\bm{\delta}_\tau$ introduced in our previous work \cite{JRT}.

\begin{definition}\label{prelimdef}
    For $v=C\ord_F\in\ValXdiv$, we may assume $C=1$. Let $\Delta$ denote the associated Okounkov body defined in \cite[\S2.5]{JRT}, and $p:\RR^n\to\RR$ the projection to the first coordinate. For $\tau\in[0,1]$, let $Q_v:=Q_v(\tau)\in[\min_\Delta p,\max_\Delta p]$ be the unique number satisfying
    $|\Delta_{\geq Q_v}|=\tau|\Delta|$, where $\Delta_{\geq t}:=\{x\in\Delta\,:\,p(x)\geq t\}$ for $t\in[\min_\Delta p,\max_\Delta p]$. Denote
    $$
        \mu_v:=p_*\left(\frac{\mu_\Delta}{\left|\Delta\right|}\right),
    $$
    where $\mu_\Delta$ is the Lebesgue measure on $\Delta$,
    $$
        \cS_\tau\left(v\right):=p\left(\Bc\Delta_{\geq Q_v\left(\tau\right)}\right),
    $$
    and
    $$
        \bm{\delta}_\tau:=\inf_{\ValXdiv}\frac{A}{S_\tau}.
    $$
\end{definition}

Similarly to Definition \ref{tildealphaDef}, these invariants must be normalized
in a delicate manner in order to encode
K/Ding-stability on the entire big cone.

\bdefn
\label{tildedeltatauDef}
Let  $\sigma:\ValXdiv\ra\RR_+$ be 
the minimal vanishing order (see \S\ref{generalized BJ}).
Set
$$\bm{\tilde\delta}_\tau:=
\begin{cases}
            \displaystyle\inf_{v\in \ValXdiv} \frac A{\cS_0+\sigma/n}=:\tilde\alpha,&\tau=0,\\
            \\
            \displaystyle\inf_{v\in\ValXdiv} \frac A{\cS_\tau+\frac{1-\tau}{\tau}\left(1-\left(1-\tau\right)^\frac{1}{n}\right)\sigma},&\tau\in\left(0,1\right].
\end{cases}
$$
\edefn
Note that in the definition $\bm{\tilde\delta}_1=\delta$ always, while on the nef cone, 
$\bm{\tilde\delta}_\tau=\bm{\delta}_\tau$ and $\tilde\alpha=\alpha$.

\begin{remark}
    We can also define the discrete version of $\widetilde{\cS}_\tau$ and $\bm{\tilde\delta}_\tau$. For $k\in\NN(L)$ and $m\in\{1,\ldots,d_k\}$, 
    denote by $\widetilde{S}_{k,m}:\ValXdiv\ra[0,\infty]$,
    $$
        \widetilde{S}_{k,m}
        :=\frac{1}{km}\sum_{\ell=1}^mj_{k,\ell}+\frac{d_k-m}{m}\left(1-\left(1-\frac{m}{d_k}\right)^\frac{1}{n}\right)\frac{j_{k,d_k}}{k},
    $$
    and
    $$
        \bm{\tilde\delta}_{k,m}:=\inf_{v\in\ValXdiv}\frac{A}{\widetilde{S}_{k,m}}.
    $$
\end{remark}

The following is the 
generalization~of~Theorem~\ref{AlphaBigThm}.

\begin{theorem}\label{GenTianThm}
Let $X$ be a normal
projective variety of dimension $n$
with klt singularities, and $L$ a big line bundle on $X$.
If for some $\tau\in[0,1]$,
$$
        \bm{\tilde\delta}_\tau\stackrel{(\ge)}{>}\begin{cases}
            \disp\frac{n}{n+1},&\tau=0;\\ \\
            \disp\tau\left(1-\left(1-\tau\right)^\frac{n+1}{n}\right)^{-1},&0<\tau\leq1,
        \end{cases}
$$
then
$\delta\stackrel{(\ge)}{>}1$, i.e., $(X,L)$ is Ding-(semi)stable.
\end{theorem}

In particular, Theorem \ref{GenTianThm} implies that 
\beq\lb{FujitaTianGenIneq}
\bm{\tilde\delta}_\tau \stackrel{(\ge)}{>} \frac n{n+1-\root n \of \tau}
\eeq
for some $\tau\in[0,1]$
implies K/Ding-(semi)stability (Remark \ref{weakerTianineqRem}),
which interpolates elegantly between Theorem \ref{AlphaBigThm} (and in particular
Odaka--Sano's Theorem \ref{OdakaSanoThm}) and Fujita--Odaka's
Theorem.

The proof of Theorem \ref{GenTianThm} relies on two ingredients.
First, a generalization of the Neumann--Hammer's barycenter inequality from convex geometry
to the setting of {\it sub-barycenters} of convex bodies
(Theorem \ref{GenHammerThm}).
Second, a generalization of Fujita's first inequality (the 
case $\tau=0$ of Theorem \ref{GenTianThm}) proved in Proposition \ref{S comparison}
(Fujita made use of the original Neumann--Hammer inequality).
Finally, we also generalize Fujita's second inequality to hold for any $\tau$ and any big $L$  
(Proposition \ref{cStauProp}) using a sort of dual inequality 
to our generalized Neumann--Hammer inequality. 
This leads to yet another generalization of Tian--Odaka--Sano's criterion, namely \eqref{FujitaTianGenIneq}.

\bigskip
\noindent
{\bf Acknowledgments.} 
Research supported in part 
by grants NSF DMS-1906370,2204347, BSF 2020329, 
NSFC 11890660, 12341105, and NKRDPC 2020YFA0712800.

\section{Generalized Neumann--Hammer theorem: sub-barycenter estimates
}\label{sec comparison}

In this section we prove a generalization of Neumann--Hammer theorem (Theorem \ref{GenHammerThm}) and use it to prove several stability criteria (Theorem \ref{GenTianThm},
Corollary \ref{GenTianBigCor}, Remark \ref{weakerTianineqRem}).

\subsection{A slice-projection generalization of the Neumann--Hammer barycenter estimate}

Let $K$ be a convex body in $\RR^n$, and let 
$\Bc K:=\int_K xdx/|K|\in\RR^n$ denote its barycenter.
Let 
$p:\RR^n\ra\RR$
denote the projection to a coordinate (i.e., $p(x)=x_i$ for some $i\in\{1,\ldots,n\}$).
A classical consequence of the Neumann--Hammer Theorem 
\cite{Neumann},\cite[Theorem II]{Hammer} is
(the result is not quite stated this way in the literature):

\begin{lemma}\label{conv lem}
    Let $K$ be a convex body in $\RR^n$.
    Then 
    $$
    \frac{n}{n+1}\min_K p+\frac{1}{n+1}\max_K p
    \le
    p(\Bc K)
    \leq 
    \frac{1}{n+1}\min_K p+\frac{n}{n+1}\max_K p.
    $$
\end{lemma}

The main result of this section is the following generalization
to sub-barycenters.
Let
$$
K_{\ge t}:=K\cap\{x\in\RR^n\,:\,p(x)\ge t\},
\q
K_{\le t}:=K\cap\{x\in\RR^n\,:\,p(x)\le t\}, \q t\in[\min_Kp,\max_Kp].
$$

\bthm\lb{GenHammerThm}
For a convex body $K\subset\RR^n$,
$$
\frac{p(\Bc K_{\ge t})-\min_Kp}{p(\Bc K)-\min_Kp}
\ge
\frac{|K|}{|K_{\ge t}|}\bigg(
1-\Big(1-\frac{|K_{\ge t}|}{|K|}\Big)^{\frac{n+1}n}
\bigg),
\q t\in \Big[\min_Kp,\max_Kp\Big),
$$
and
$$
\frac{\max_Kp-p(\Bc K_{\le t})}{\max_Kp-p(\Bc K)}
\ge
\frac{|K|}{|K_{\le t}|}\bigg(
1-\Big(1-\frac{|K_{\le t}|}{|K|}\Big)^{\frac{n+1}n}
\bigg),
\q t\in \Big(\min_Kp,\max_Kp\Big].
$$
\ethm
Lemma \ref{conv lem}, rephrased as 
$\frac{1}{n+1}\osc_K p\le p(\Bc K)-\min_Kp\leq \frac{n}{n+1}\osc_K p$,
follows from Theorem \ref{GenHammerThm}
in the limits 
$t\ra \max p(K)$ of the first inequality (gives the upper bound) and $t\ra \min p(K)$ of the second (gives the lower bound). 

\begin{proof}[Proof of Theorem \ref{GenHammerThm} (assuming Proposition \ref{concave comparison})]
Note that 
$$
p(\Bc K_{\ge t}))
=
|K_{\ge t}|^{-1}\int_t^\infty s\big|p^{-1}(s)\cap K\big|ds=
\frac{\disp\int_t^\infty s\big|p^{-1}(s)\cap K\big|ds}
{\disp\int_t^\infty \big|p^{-1}(s)\cap K\big|ds}.
$$
By convexity of $K$, for $\lambda\in(0,1)$,
$(1-\lambda)p^{-1}(s)\cap K+\lambda p^{-1}(s')\cap K
\subset p^{-1}((1-\lambda)s+\lambda s')\cap K$.
Together with Brunn--Minkowski,
\begin{align*}
    \left(1-\lambda\right)\Big|p^{-1}(s)\cap K\Big|^\frac{1}{n-1}
+
\lambda\Big|p^{-1}(s')\cap K\Big|^\frac{1}{n-1}
&\le
\Big|\left(1-\lambda\right)p^{-1}(s)\cap K
+
\lambda p^{-1}(s')\cap K
\Big|^\frac{1}{n-1}\\
&\le
\Big|p^{-1}\Big(\left(1-\lambda\right)s+\lambda s'\Big)\cap K\Big|^\frac{1}{n-1}.
\end{align*}
In other words,
$s\to|p^{-1}(s)\cap K|^\frac{1}{n-1}$ is a concave function on $p(K)$.
Thus, the first inequality of Theorem \ref{GenHammerThm} is a special case
of Proposition \ref{concave comparison} 
for
$f(s):=|p^{-1}(s)\cap K\big|^{\frac1{n-1}}$,
with
$T=\osc_Kp$
and $t\in [0,\osc_Kp)$ (since for a convex body $K$
and for $t$ in this range $|K_{\ge t}|>0$).
The second inequality follows from 
Proposition \ref{concave comparison} applied to
the concave function $s\mapsto f(\osc_Kp-s)$ on $[0,\osc_Kp]$,
followed by the change of variable
$s\mapsto \osc_Kp-s$ with $s\in(0,\osc_Kp]$.
\epf

\subsection{Generalized functional Neumann--Hammer inequality}

The functional version of Theorem \ref{GenHammerThm} is: 

\bprop\label{concave comparison}
Fix $T>0$ and $n\in\NN$.
Let $f:[0,T]\to\RR_{\geq0}$ be concave
and not identically zero.
 Then
$$
 \frac{\disp\int_{t}^Tsf\left(s\right)^{n-1}ds}
 {\disp\int_0^Tsf\left(s\right)^{n-1}ds}\geq1-\left(1-\frac{\disp\int_{t}^Tf(s)^{n-1}ds}{\disp\int_0^Tf(s)^{n-1}ds}\right)^\frac{n+1}{n},
     \q  \hbox{\ for $t\in[0,T]$}.
$$
\eprop

Denote
\beq\lb{VtNotationEq}
V_{\le t}:=\int_{0}^tf(s)^{n-1}ds,
\q
V_{\ge t}:=\int_{t}^Tf(s)^{n-1}ds,
\q t\in[0,T],
\eeq
and
\beq\lb{btNotation}
b_{\le t}:=
\frac{\disp\int_0^{t}sf\left(s\right)^{n-1}ds}{V_{\le t}},
\q
b_{\ge t}:=\frac{\disp\int_t^{T}sf\left(s\right)^{n-1}ds}{V_{\ge t}},
\q
t\in(0,T).
\eeq
It will be convenient to work with the volume ratio
\beq\lb{taulemmaEq}
 \tau\equiv
 \tau_{\ge t}:=
 \frac{\disp\int_{t}^Tf(s)^{n-1}ds}{\disp\int_0^Tf(s)^{n-1}ds}=\frac{V_{\ge t}}{V_{\ge0}},
 \q  \hbox{\ for $t\in[0,T]$}.
\eeq

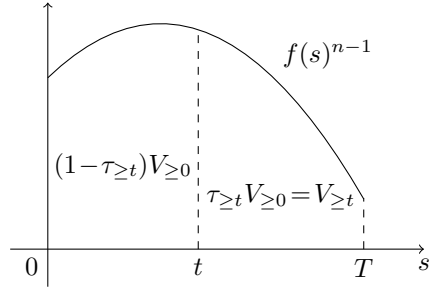
\begin{figure}[ht]
    \centering
        \begin{tikzpicture}
            \draw[->](-.5,0)--(5,0)node[below]{$s$};
            \draw[->](0,-.5)--(0,3.28);
            \draw(0,0)node[below left]{$0$};
            \draw[domain=0:4.2]plot(\x,{3-.32*(\x-1.5)^2});
            \draw(3,2.28)node[above right]{$f(s)^{n-1}$};
            \draw[dashed](4.2,0)node[below]{$T$}--(4.2,0.68);
            \draw[dashed](2,0)node[below]{$t$}--(2,2.92);
            \draw(3.1,0.7)node{$\tau_{\ge t} V_{\ge0}\!=\!V_{\ge t}$};
            \draw(1,1.1)node{$(1\!-\!\tau_{\ge t})V_{\ge0}$};
        \end{tikzpicture}
    \caption{An illustration of Proposition \ref{concave comparison}.} 
\end{figure}

\begin{proof}[Proof of Proposition \ref{concave comparison}]
    When $t=0$ the inequality holds. Fix $t\in(0,T]$. Then $\tau<1$. For $s\in[0,(1-\tau)^{-\frac{1}{n}}T]$, consider the auxiliary function
    $$
        F\left(s\right):=\left(1-\tau\right)^{-\frac{1}{n}}f\left(\left(1-\tau\right)^\frac{1}{n}s\right).
    $$
    Since $f$ is concave and $f(0)\geq0$, for $s\in[0,T]$,
    $$
        f\left(\left(1-\tau\right)^\frac{1}{n}s\right)\geq\left(1-\left(1-\tau\right)^\frac{1}{n}\right)f\left(0\right)+\left(1-\tau\right)^\frac{1}{n}f\left(s\right)\geq\left(1-\tau\right)^\frac{1}{n}f\left(s\right),
    $$
    i.e.,
    \begin{equation}\label{Ff comparison}
        F\left(s\right)\geq f\left(s\right).
    \end{equation}
    Now, change variable to get
    \begin{align}
        \int_0^{\left(1-\tau\right)^{-\frac{1}{n}}t}F\left(s\right)^{n-1}ds&=\left(1-\tau\right)^{-\frac{1}{n}}\int_0^tF\left(\left(1-\tau\right)^{-\frac{1}{n}}s\right)^{n-1}ds\nonumber\\
        &=\frac{1}{1-\tau}\int_0^tf\left(s\right)^{n-1}ds
        =V_{\geq0}\nonumber\\
        &=\int_0^Tf\left(s\right)^{n-1}ds\label{V reverse}\\
        &\leq\int_0^TF\left(s\right)^{n-1}ds,\nonumber
    \end{align}
    i.e.,
    \begin{equation}\label{tT comparison}
        \left(1-\tau\right)^{-\frac{1}{n}}t\leq T.
    \end{equation}
    Changing variable again and applying \eqref{Ff comparison}, \eqref{V reverse}, and \eqref{tT comparison},
    \begin{align*}
        \left(1-\tau\right)^{-\frac{n+1}{n}}\int_0^tsf\left(s\right)^{n-1}ds&=\left(1-\tau\right)^{-\frac{n-1}{n}}\int_0^{\left(1-\tau\right)^{-\frac{1}{n}}t}sf\left(\left(1-\tau\right)^\frac{1}{n}s\right)^{n-1}ds\\
        &=\int_0^{\left(1-\tau\right)^{-\frac{1}{n}}t}sF\left(s\right)^{n-1}ds\\
        &=\int_0^{\left(1-\tau\right)^{-\frac{1}{n}}t}sf\left(s\right)^{n-1}ds+\int_0^{\left(1-\tau\right)^{-\frac{1}{n}}t}s\left(F\left(s\right)^{n-1}-f\left(s\right)^{n-1}\right)ds\\
        &\begin{multlined}
            \leq\int_0^{\left(1-\tau\right)^{-\frac{1}{n}}t}sf\left(s\right)^{n-1}ds\\+\left(1-\tau\right)^{-\frac{1}{n}}t\int_0^{\left(1-\tau\right)^{-\frac{1}{n}}t}\left(F\left(s\right)^{n-1}-f\left(s\right)^{n-1}\right)ds
        \end{multlined}\\
        &\begin{multlined}
            =\int_0^{\left(1-\tau\right)^{-\frac{1}{n}}t}sf\left(s\right)^{n-1}ds\\+\left(1-\tau\right)^{-\frac{1}{n}}t\left(\int_0^Tf\left(s\right)^{n-1}ds-\int_0^{\left(1-\tau\right)^{-\frac{1}{n}}t}f\left(s\right)^{n-1}ds\right)
        \end{multlined}\\
        &=\int_0^{\left(1-\tau\right)^{-\frac{1}{n}}t}sf\left(s\right)^{n-1}ds+\left(1-\tau\right)^{-\frac{1}{n}}t\int_{\left(1-\tau\right)^{-\frac{1}{n}}t}^Tf\left(s\right)^{n-1}ds\\
        &\leq\int_0^{\left(1-\tau\right)^{-\frac{1}{n}}t}sf\left(s\right)^{n-1}ds+\int_{\left(1-\tau\right)^{-\frac{1}{n}}t}^Tsf\left(s\right)^{n-1}ds\\
        &=\int_0^Tsf\left(s\right)^{n-1}ds,
    \end{align*}
    i.e.,
    $$
        \frac{\disp\int_0^tsf\left(s\right)^{n-1}ds}{\disp\int_0^Tsf\left(s\right)^{n-1}ds}\leq\left(1-\tau\right)^\frac{n+1}{n}.
    $$
    This completes the proof.
\end{proof}

\begin{remark}
    Let $g\not\equiv0$ be a non-negative non-decreasing integrable function on $[0,T]$. A variant of the proof will yield
    $$
        \frac{\disp\int_{t}^Tg\left(s\right)f\left(s\right)^{n-1}ds}
        {\disp\int_0^Tg\left(s\right)f\left(s\right)^{n-1}ds}\geq1-\frac{\disp\int_{0}^ts^{n-1}g\left(s\right)ds}{\disp\int_0^{\tilde{t}}s^{n-1}g\left(s\right)ds},
        \q  \hbox{\ for $t\in[0,T]$},
    $$
    with $\tilde{t}$ given by the left hand side of \eqref{tT comparison}.
    For example, if $g(s)=s^p$ for $p\geq0$, then
    \beq\lb{genZhangIneq}
        \frac{\disp\int_{t}^Ts^pf\left(s\right)^{n-1}ds}
        {\disp\int_0^Ts^pf\left(s\right)^{n-1}ds}\geq1-\left(1-\tau\right)^\frac{n+p}{n},
    \eeq
     This is related to work of Zhang, who considered a different, neat way to interpolate $\alpha$ and $\delta$ (though he does not consider $\alpha_k$, $\delta_k$, nor asymptotics as in \cite{JRT}), by studying $\ell^p$ norms of the jumping function, $p=1$ giving $\delta$ and $p=\infty$ giving $\alpha$. Our methods, together with \eqref{genZhangIneq} (in the limit $\tau\ra0$), yield asymptotics for a quantized
     version of his invariants $\delta^{(p)}$ (or more
     general invariants parametrized by both $p$ and $\tau$), as well as
    generalizations of~his~main~results~(e.g.,~\cite[Theorem~1.8]{Zhang}).
\eremark

\section{A generalized Fujita inequality of the first kind}
\label{generalized BJ}
Fujita's first barycetner inequality \cite[Proposition 2.1]{Fuj19},
in the generality stated by Blum--Jonsson \cite[Proposition 3.11]{BJ20}, 
assumes $L$ is {\it ample} and reads
$$
\cS_0\geq\frac{n+1}{n}\cS_1, \q \hbox{ on $\ValXfin$}.
$$
The main result of this subsection 
is the following generalization to sub-barycenters.
We also allow $L$ to be big.
Recall that
$$\sigma(v)=\min_{\Delta}p_1.$$
For $v\in\ValXdiv$, this definition goes back at least to
Nakayama \cite[Definition 1.1, p. 79]{Nak}.

\begin{proposition}\label{S comparison}
    On $\ValXdiv$,
    $$
        \cS_\tau\geq\begin{cases}
            \displaystyle\frac{n+1}{n}\cS_1-\frac1n\sigma,&\tau=0,\\
            \\
            \displaystyle\frac{1}{\tau}\left(1-\left(1-\tau\right)^\frac{n+1}{n}\right)(\cS_1-\sigma)+\sigma,&\tau\in\left(0,1\right].
        \end{cases}
    $$
\end{proposition}
\begin{proof}
    Recall Definition \ref{prelimdef}. We can write $d\mu_v=f(s)^{n-1}ds$ for some concave function $f$ on $[\sigma(v),\cS_0(v)]$. 
    Also, $Q_v$ now takes values in $[\sigma(v),\cS_0(v)]$, and 
    $$
        \int_{\sigma(v)}^{\cS_0\left(v\right)}d\mu_v={{\VolL}/{n!}}, \q
        \int_{Q_v\left(\tau\right)}^{\cS_0\left(v\right)}d\mu_v=\tau{{\VolL}/{n!}}.
    $$
    Translating $\Delta$
    by $(\sigma(v),0,\ldots,0)$ amounts to the change of variable $s\!\mapsto \!s-\sigma(v)$. 
    By~Proposition~\ref{concave comparison},
    $$
        \frac{\cS_\tau\left(v\right)-\sigma(v)}{\cS_1\left(v\right)-\sigma(v)}
        =
        \frac{1}{\tau}\frac{\int_{Q_v\left(\tau\right)-\sigma(v)}^{\cS_0\left(v\right)-\sigma(v)}td\mu_v}
        {\int_{0}^{\cS_0\left(v\right)-\sigma(v)}td\mu_v}\geq\frac{1}{\tau}\left(1-\left(1-\tau\right)^\frac{n+1}{n}\right),\q
        \hbox{\ for $\tau\in(0,1]$}.
    $$
By monotonicity of $\cS_\tau$, 
    $$
        \frac{\cS_0\left(v\right)-\sigma(v)}{\cS_1\left(v\right)-\sigma(v)}\geq\frac{\cS_\tau\left(v\right)-\sigma(v)}{\cS_1\left(v\right)-\sigma(v)}\geq\frac{1}{\tau}\left(1-\left(1-\tau\right)^\frac{n+1}{n}\right),\q
        \hbox{\ for $\tau\in(0,1]$}.
    $$
    Finally, letting $\tau$ tend to $0$,
    $\disp
        \frac{\cS_0\left(v\right)-\sigma(v)}{\cS_1\left(v\right)-\sigma(v)}\geq\frac{n+1}{n}.
    $
    \end{proof}
Note that Proposition \ref{S comparison} implies that
$$
\frac{\cS_0}A+\frac1n\frac{\sigma}A\le \frac{n+1}n\q\Rightarrow\q \frac{\cS_1}A\le 1.
$$
In particular, defining
$$
\tilde\alpha:=\inf_{v\in\ValXdiv} \frac A{\cS_0+\sigma/n},
$$
gives the following generalization of Tian--Odaka--Sano's criterion to big $L$:
\bcor\lb{GenTianBigCor}
If $\tilde\alpha\stackrel{(\ge)}{>} \frac n{n+1}$ then $\delta\stackrel{(\ge)}{>}1$.
\ecor

\section{A generalized Fujita inequality of the second kind}
\label{generalized-2nd-Fuj}

Fujita's second barycenter inequality \cite[Lemma 2.2]{Fuj19}
reads
\beq\lb{Fuj2ndInEq}
\cS_1\ge\frac1{n+1}\cS_0, \q \hbox{ on $\ValXdiv$}.
\eeq
This generalizes to sub-barycenters as well.

\bprop
\lb{cStauProp}
For $\tau\in[0,1]$,
$\cS_\tau\ge (1-\tau^{1/n})\cS_0+\tau^{1/n}\cS_1$,
on $\ValXdiv$.
\eprop

To see why Proposition \ref{cStauProp}
is a generalization of \eqref{Fuj2ndInEq}, note that $\tau \cS_\tau$
is non-decreasing in $\tau$. In particular, $\cS_1\ge \tau \cS_\tau$.
Thus, Proposition \ref{cStauProp} implies 
$$
\cS_1\ge \frac{1-\tau^{1/n}}{\tau^{-1}-\tau^{1/n}}\cS_0, \q \tau\in(0,1).
$$
Note that $x\mapsto (1-x^{1/n})/(x^{-1}-x^{1/n})$ is increasing
on $(0,1)$, and its limit as $x\ra1$ is, by L'H\^opital's rule, 
$1/(n+1)$, implying \eqref{Fuj2ndInEq}.

\bpf[Proof of Proposition \ref{cStauProp}]
Assume first $\sigma(v)=0$.
Applying Proposition \ref{concave comparison} to $s\mapsto f(T-s)$, yields
$$
 \frac{\disp\int_{t}^Tsf\left(T-s\right)^{n-1}ds}
 {\disp\int_0^Tsf\left(T-s\right)^{n-1}ds}\geq1-\left(1-\frac{\disp\int_{t}^Tf(T-s)^{n-1}ds}{\disp\int_0^Tf(T-s)^{n-1}ds}\right)^\frac{n+1}{n},
     \q  \hbox{\ for $t\in[0,T]$}.
$$
Now, change variable to get
$$
 \frac{\disp\int_{0}^{T-t}(T-s)f\left(s\right)^{n-1}ds}
 {\disp\int_0^T(T-s)f\left(s\right)^{n-1}ds}\geq1-\left(1-\frac{\disp\int_{0}^{T-t}f(s)^{n-1}ds}{\disp\int_0^Tf(s)^{n-1}ds}\right)^\frac{n+1}{n},\q  \hbox{\ for $t\in[0,T]$},
$$
or
$$
 \frac{\disp\int_{0}^{t}(T-s)f\left(s\right)^{n-1}ds}
 {\disp\int_0^T(T-s)f\left(s\right)^{n-1}ds}\geq1-\left(1-\frac{\disp\int_{0}^{t}f(s)^{n-1}ds}{\disp\int_0^Tf(s)^{n-1}ds}\right)^\frac{n+1}{n},\q  \hbox{\ for $t\in[0,T]$},
$$
i.e., by \eqref{taulemmaEq},
$$
 \frac{\disp TV_{\le t}-\int_{0}^{t}sf\left(s\right)^{n-1}ds}
 {\disp TV_{\ge0}-\int_0^Tsf\left(s\right)^{n-1}ds}\geq1-\tau^\frac{n+1}{n},\q  \hbox{\ for $t\in[0,T]$}.
$$
Some further calculation yield combined with Definition \ref{prelimdef},
$$
\frac{(1-\tau)\cS_0+\tau\cS_\tau-\cS_1}{\cS_0-\cS_1}
\ge
{1-\tau^{\frac{n+1}n}},
$$
as desired. Since the inequality obtained is invariant under translation in
the $s$ variable, it holds even for $\sigma(v)>0$.
\epf

Several authors observed that Fujita's second inequality \eqref{Fuj2ndInEq} holds more generally
for $L$ merely big (as well as on $\ValXfin$) \cite[Lemma 2.6, Proposition 3.11, Remark 3.12]{BJ20}.
It is perhaps not entirely without interest to note an even stronger statement holds.

\bcor\lb{GenFujSecSigma}
On $\ValXfin$,
\beq\lb{FujitaNakIneq}
\cS_1
\ge
\frac1{n+1}\cS_0+\frac n{n+1}\sigma.
\eeq
\ecor

\bpf
 In fact, notice that the inequality in Proposition \ref{cStauProp} is invariant if we subtract a number from each $\cS_\tau$. This amounts to translating $\Delta$. 
Thus, apply the same argument above that yields \eqref{Fuj2ndInEq} from Proposition \ref{cStauProp}
to $\tilde\cS_\tau:=\cS_\tau-\sigma$ to obtain \eqref{FujitaNakIneq} on $\ValXdiv$.
Alternatively, this follows directly from Lemma \ref{conv lem}. The extension to $\ValXfin$
then follows from the methods of \cite{BJ20}.
\epf

\bremark
\lb{weakerTianineqRem}
For $L$ ample, using $\cS_0\ge\frac{n+1}n\cS_1$, Proposition \ref{cStauProp} implies, in particular, that
$$
\bm{\delta}_\tau\ge\frac{n}{n+1-\tau^{1/n}}
$$
implies $\delta\ge1$. For $L$ big, by Proposition \ref{S comparison}, the same conclusion holds for $\bm{\tilde\delta}_\tau$
\eqref{FujitaTianGenIneq}. Note that for $\tau\in(0,1)$ this is weaker
than Theorem \ref{GenTianThm}.
\eremark

\subsection{Proof of Theorem \ref{GenTianThm}}
\label{generalized Tian}
\begin{proof}
    By Proposition \ref{S comparison}, on $\ValXdiv$,
    $$
        \cS_1\leq\begin{cases}
            \displaystyle\frac{n}{n+1}\left(\cS_0+\sigma/n\right),&\tau=0,\\
            \\
            \displaystyle\tau\left(1-\left(1-\tau\right)^\frac{n+1}{n}\right)^{-1}\left(\cS_\tau+\frac{1-\tau}{\tau}\left(1-\left(1-\tau\right)^\frac{1}{n}\right)\right),&\tau\in\left(0,1\right].
        \end{cases}
    $$
    Combining with Definition \ref{tildedeltatauDef},
    $$
        \delta=\inf_{v\in\ValX}\frac{A}{\cS_1}\geq\begin{cases}
            \displaystyle\frac{n+1}{n}\bm{\tilde\delta}_0,&\tau=0,\\
            \\
            \displaystyle\frac{1}{\tau}\left(1-\left(1-\tau\right)^\frac{n+1}{n}\right)\bm{\tilde\delta}_\tau,&\tau\in\left(0,1\right].
        \end{cases}
    $$
    This completes the proof.
\end{proof}

\section{K-stability of cubic surfaces with an Eckardt point}

The Calabi Conjecture for smooth del Pezzo surfaces was established by Tian in 1990 \cite{Tian90-2}.
The proof consisted of two main parts. First, computing the $\alpha$-invariant
and showing it is greater than $2/3$ unless the surface is a cubic surface in $\PP^3$
with an Eckardt point. Second, treating the latter special surfaces 
required deep analysis consisting of Tian's celebrated strong partial $C^0$-estimates
and more
complicated invariants introduced by Tian denoted $\alpha_{k,2}$ that he was 
able to estimate along a subsequence $\alpha_{6k,2}>2/3$.
This last estimate
was perhaps the most difficult technical part of the proof and consisted of a tour de force.
Twenty years later, Shi was able to simplify this estimate  
by showing that actually $\alpha_{k,2}>2/3$ (not just along a subsequence) \cite{Shi10},
which allows to rely only on a weaker version of the partial $C^0$-estimate. 
Nevertheless, even
with Shi's simplification the proof remains extremely difficult. 
In this section, we show how the $\bm{\delta}_\tau$
invariants and Theorem \ref{GenTianThm} give a simplified and conceptual
proof. The simplification is roughly in that in using $\tau$-percentage of sections with any
$\tau>0$ the computations are simpler than when taking just two sections as in $\alpha_{k,2}$,
which is a rather more degenerate situation (as $\lim_k 2/k = 0$).

    Let $S$ be a smooth cubic surface with an Eckardt point $p\in S$, i.e., there are 3 lines
    (by line we mean an intersection of a hyperplane in $\PP^3$ with $S$) on $S$ that contain $p$. Let $\sigma:\Bl_pS\to S$ denote the blow up at $p$ and $E$ the exceptional divisor on $\Bl_pS$. 
    For $v=\ord_E$ \cite[p. 71]{Cheltsov-book},
    \begin{equation}\label{A}
        A\left(v\right)=2,
    \end{equation}
    and
    $$
        \left|\Delta_{\geq t}\right|=\begin{cases}
            3-t^2,&t\in\left[0,1\right],\\
            \displaystyle\frac{\left(3-t\right)^2}{2},&t\in\left(1,3\right].
        \end{cases}
    $$
    Recall Definition \ref{prelimdef},
    $$
        \mu_v=\begin{cases}
            \displaystyle\frac{2}{3}tdt,&t\in\left[0,1\right],\\
            \displaystyle\left(1-\frac{t}{3}\right)dt,&t\in\left(1,3\right].
        \end{cases}
    $$
    The volume quantile is then
    \begin{equation}\label{tau(t)}
        \tau:=\tau\left(t\right)=\int_t^3d\mu_v=\frac{\left|\Delta_{\geq t}\right|}{\left|\Delta\right|}=\begin{cases}
            \displaystyle1-\frac{t^2}{3},&t\in\left[0,1\right],\\
            \displaystyle\frac{\left(3-t\right)^2}{6},&t\in\left(1,3\right],
        \end{cases}
    \end{equation}
    and (recall Definition \ref{prelimdef})
    $$
        \tau\cS_\tau\left(v\right)=\int_t^3sd\mu_v\left(s\right).
    $$
    When $t\in(1,3]$,
    $$
        \int_t^3sd\mu_v\left(s\right)=\int_t^3\frac{s\left(3-s\right)}{3}ds\xlongequal{r=3-s}\int_0^{3-t}\left(r-\frac{r^2}{3}\right)dr=\left.r^2\left(\frac{1}{2}-\frac{r}{9}\right)\right|_{r=0}^{3-t}=\frac{\left(3-t\right)^2\left(2t+3\right)}{18}.
    $$
    In particular,
    $$
        \int_1^3sd\mu_v\left(s\right)=\frac{10}{9}.
    $$
    When $t\in[0,1]$,
    $$
        \int_t^3sd\mu_v\left(s\right)=\int_t^1\frac{2}{3}s^2ds+\frac{10}{9}=\left.\frac{2}{9}s^3\right|_{s=t}^1+\frac{10}{9}=\frac{4}{3}-\frac{2}{9}t^3.
    $$
    In summary,
    $$
        \tau\cS_\tau\left(v\right)=\begin{cases}
            \displaystyle\frac{4}{3}-\frac{2}{9}t^3,&t\in\left[0,1\right],\\
            \displaystyle\frac{\left(3-t\right)^2\left(2t+3\right)}{18},&t\in\left(1,3\right].
        \end{cases}
    $$
    Solving for $t$ in \eqref{tau(t)},
    $$
        t=\begin{cases}
            \displaystyle3-\sqrt{6\tau},&\displaystyle\tau\in\left[0,\frac{2}{3}\right),\\
            \displaystyle\sqrt{3\left(1-\tau\right)},&\displaystyle\tau\in\left[\frac{2}{3},1\right].
        \end{cases}
    $$
    Therefore, when $\tau\in[0,\frac{2}{3})$, i.e., $t\in(1,3]$,
    $$
        \cS_\tau\left(v\right)=\frac{\left(3-t\right)^2\left(2t+3\right)}{18\tau}=\frac{2}{3}t+1=3-\frac{2}{3}\sqrt{6\tau}.
    $$
    When $\tau\in[\frac{2}{3},1]$, i.e., $t\in[0,1]$,
    $$
        \cS_\tau\left(v\right)=\frac{2}{9\tau}\left(6-t^3\right)=\frac{2}{3\tau}\left(2-\sqrt{3}\left(1-\tau\right)^\frac{3}{2}\right).
    $$
    In summary,
    $$
        \cS_\tau\left(v\right)=\begin{cases}
            \displaystyle3-\frac{2}{3}\sqrt{6\tau},&\displaystyle\tau\in\Big[0,\frac{2}{3}\Big),\\
            \displaystyle\frac{2}{3\tau}\left(2-\sqrt{3}\left(1-\tau\right)^\frac{3}{2}\right),&\displaystyle\tau\in \Big[\frac{2}{3},1\Big].
        \end{cases}
    $$
    By \eqref{A},
    $$
        \frac{A\left(v\right)}{\cS_\tau\left(v\right)}=\begin{cases}
            \displaystyle\frac{6}{9-2\sqrt{6\tau}},&\displaystyle\tau\in\Big[0,\frac{2}{3}\Big),\\
            \displaystyle3\tau\left(2-\sqrt{3}\left(1-\tau\right)^\frac{3}{2}\right)^{-1},&\displaystyle\tau\in\Big[\frac{2}{3},1\Big].
        \end{cases}
    $$

    \begin{figure}[ht]
        \centering
        \begin{tikzpicture}
            \draw(0,0)node[below left]{$0$};
            \draw[->](-.5,0)--(6.5,0)node[below]{$\tau$};
            \draw(6,0)node[below]{$1$};
            \draw[->](0,-.5)--(0,3.5);
            \filldraw(0,4/3)circle(1pt)node[left]{$\frac{2}{3}$};
            \filldraw(0,2)circle(1pt)node[left]{$1$};
            \filldraw(0,3)circle(1pt)node[left]{$\frac{3}{2}$};
            \draw[domain=0:4]plot(\x,{12/(9-2*sqrt(\x))});
            \draw[domain=4:6]plot(\x,{\x/(2-sqrt(3)*(1-\x/6)^1.5)});
            \draw[domain=1e-2:6,dashed]plot(\x,{\x/(3*(1-(1-\x/6)^1.5))});
            \draw[dotted](6,0)--(6,3);
        \end{tikzpicture}
        \caption{The graph of $\tau\mapsto\frac{A(v)}{\cS_\tau(v)}$. The dashed line is the lower bound for Theorem \ref{GenTianThm} to apply.}\label{compare}
    \end{figure}
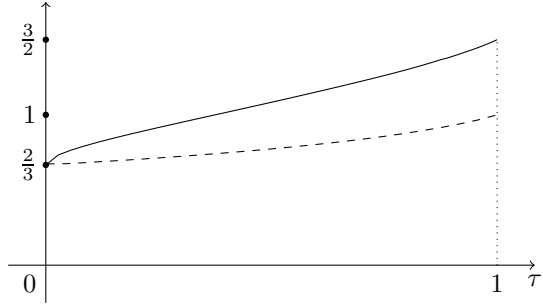

By Cheltsov's work, both $\alpha$ and $\delta$ are computed
by $v$ \cite[p. 71]{Cheltsov-book},\cite{Cheltsov08}; similar arguments can be used to show the same is true for all $\bm{\delta}_\tau$. Thus,
    for {\it any } $\tau\in(0,1]$, Theorem \ref{GenTianThm} implies that $\delta>1$. Indeed, when $\tau\in(0,\frac{2}{3})$, using Taylor expansion with Lagrange remainder, there is $\xi\in[0,\tau]$ such that
    $$
        \left(1-\tau\right)^\frac{3}{2}=1-\frac{3}{2}\tau+\frac{3}{8\sqrt{1-\xi}}\xi^2<1-\frac{3}{2}\tau+\frac{3\sqrt{3}}{8}\tau^2<1-\frac{3}{2}\tau+\frac{\sqrt{6}}{3}\tau^2<1-\frac{3}{2}\tau+\frac{\sqrt{6}}{3}\tau\sqrt{\tau},
    $$
    i.e.,
    $$
        \frac{\tau}{1-\left(1-\tau\right)^\frac{3}{2}}<\frac{6}{9-2\sqrt{6\tau}}.
    $$
    When $\tau\in[\frac{2}{3},1]$, since
    $$
        \left(3-\sqrt{3}\right)\left(1-\tau\right)^\frac{3}{2}\leq\left(3-\sqrt{3}\right)\cdot3^{-\frac{3}{2}}=\frac{\sqrt{3}-1}{3}<1,
    $$
    we have
    $$
        \frac{\tau}{1-\left(1-\tau\right)^\frac{3}{2}}=\frac{3\tau}{2-3\left(1-\tau\right)^\frac{3}{2}+1}<\frac{3\tau}{2-3\left(1-\tau\right)^\frac{3}{2}+\left(3-\sqrt{3}\right)\left(1-\tau\right)^\frac{3}{2}}=\frac{3\tau}{2-\sqrt{3}\left(1-\tau\right)^\frac{3}{2}}.
    $$
    See Figure \ref{compare}.

\bigskip
\bigskip
\textsc{University of Maryland}

\smallskip
{\tt cjin123@terpmail.umd.edu, yanir@alum.mit.edu}

\smallskip
\textsc{Peking University}

\smallskip
{\tt gtian@math.pku.edu.cn}

\end{document}